\documentclass[a4paper,11pt]{article}
\usepackage{amsmath,amsthm,amssymb,enumitem,xcolor}
\usepackage{tikz}
\usetikzlibrary{shapes,arrows,decorations.markings,positioning}

\usepackage[nosort,nocompress]{cite}

\usepackage[bookmarks=false,hyperfootnotes=false,colorlinks,
    linkcolor={red!60!black},
    citecolor={blue!50!black},
    urlcolor={blue!80!black}]{hyperref}

\renewcommand{\eqref}[1]{\hyperref[#1]{(\ref{#1})}}

\newcommand{\StrikeThruDistance}{0.6em}

\tikzstyle{property}=[draw, ellipse,thick]
\tikzstyle{rightarrow}=[-implies,double distance=0.25em,shorten <=0.2em,shorten >=0.1em,thick]
\tikzstyle{rightarrow strike thru}=[%
    rightarrow,
    decoration={%
        markings, mark=at position 0.5 with {
            \draw[ultra thick,solid,-] ++(-\StrikeThruDistance,-\StrikeThruDistance) -- (\StrikeThruDistance,\StrikeThruDistance);
        }
    },
    postaction={decorate}
]
\tikzstyle{rightarrow question}=[%
    rightarrow,
    densely dotted,
    decoration={
        markings, mark=at position 0.5 with {
            \draw node {\LARGE \textbf{?}};
        }
    },
    postaction={decorate}
]

\newlist{enumlist}{enumerate}{1}
\setlist[enumlist]{labelindent=0cm,label=\arabic*.,ref=\arabic*,labelwidth=2.5ex,labelsep=0.5ex,leftmargin=3ex,align=left,topsep=0.5ex,itemsep=1ex,parsep=1ex}

\newlist{itemlist}{itemize}{1}
\setlist[itemlist]{labelindent=0cm,label=$\bullet$,labelwidth=2.5ex,labelsep=0.5ex,leftmargin=3ex,align=left,topsep=0.5ex,itemsep=1ex,parsep=1ex}

\newlist{lijst}{itemize}{1}
\setlist[lijst]{labelindent=0cm,label={},labelwidth=4ex,labelsep=1ex,leftmargin=5ex,align=left,topsep=0.5ex,itemsep=1ex,parsep=1ex}

\numberwithin{equation}{section}

{\theoremstyle{definition}\newtheorem{definition}{Definition}[section]

}

\newtheorem{proposition}[definition]{Proposition}
\newtheorem{lemma}[definition]{Lemma}
\newtheorem{theorem}[definition]{Theorem}

\newcommand{\C}{\mathbb{C}}

\newcommand{\al}{\alpha}
\newcommand{\be}{\beta}

\newcommand{\recht}{\rightarrow}
\newcommand{\Z}{\mathbb{Z}}

\newcommand{\om}{\omega}
\newcommand{\N}{\mathbb{N}}

\newcommand{\ovt}{\mathbin{\overline{\otimes}}}

\newcommand{\R}{\mathbb{R}}
\newcommand{\F}{\mathbb{F}}

\newcommand{\cZ}{\mathcal{Z}}
\newcommand{\Ad}{\operatorname{Ad}}

\newcommand{\cK}{\mathcal{K}}

\newcommand{\T}{\mathbb{T}}
\newcommand{\actson}{\curvearrowright}

\newcommand{\cU}{\mathcal{U}}

\newcommand{\cN}{\mathcal{N}}
\newcommand{\cR}{\mathcal{R}}

\newcommand{\dis}{\displaystyle}

\newcommand{\SL}{\operatorname{SL}}
\newcommand{\bI}{\mathbb{I}}

\begin{document}

\begin{center}
{\boldmath\LARGE\bf Inner amenability, property Gamma,\vspace{0.5ex}\\ McDuff II$_1$ factors and stable equivalence relations}

\bigskip

{\sc by Tobe Deprez\footnote{KU~Leuven, Department of Mathematics, Leuven (Belgium), tobe.deprez@kuleuven.be \\ Supported by a PhD fellowship of the Research Foundation Flanders (FWO).} and Stefaan Vaes\footnote{KU~Leuven, Department of Mathematics, Leuven (Belgium), stefaan.vaes@kuleuven.be \\
    Supported in part by European Research Council Consolidator Grant 614195, and by long term structural funding~-- Methusalem grant of the Flemish Government.}}

\end{center}

\begin{abstract}\noindent
We say that a countable group $G$ is McDuff if it admits a free ergodic probability measure preserving action such that the crossed product is a McDuff II$_1$ factor. Similarly, $G$ is said to be stable if it admits such an action with the orbit equivalence relation being stable. The McDuff property, stability, inner amenability and property Gamma are subtly related and several implications and non implications were obtained in \cite{Ef73,JS85,Va09,Ki12a,Ki12b}. We complete the picture with the remaining implications and counterexamples.
\end{abstract}

\section{Introduction}

Murray and von Neumann proved in \cite{MvN43} that the hyperfinite II$_1$ factor $R$ is not isomorphic to the free group factor $L(\F_2)$, by showing that $R$ has \emph{property Gamma} while $L(\F_2)$ has not. Recall that a tracial von Neumann algebra $(M,\tau)$ has property Gamma if there exists a net of unitaries $u_n \in M$ such that $u_n \recht 0$ weakly and $\lim_n \|x u_n - u_n x \|_2 = 0$ for all $x \in M$, where $\|x\|_2 = \sqrt{\tau(x^* x)}$.

We say that a countable group $G$ has \emph{property Gamma} when $L(G)$ does. In \cite{Ef73}, it was shown that if $G$ has property Gamma, then $G$ is \emph{inner amenable~:} there exists a conjugation invariant mean $m$ on $G$ such that $m(\cU) = 0$ for every finite subset $\cU \subset G$. The converse does not hold~: an inner amenable group $G$ with infinite conjugacy classes (icc) but not having property Gamma was constructed in \cite{Va09}.

If a II$_1$ factor $M$ admits non commuting central sequences, more precisely if the central sequence algebra $M' \cap M^\om$ is non abelian, then $M \cong M \ovt R$ and $M$ is said to be \emph{McDuff}, see \cite{McD69}. The counterpart of the McDuff property for equivalence relations was introduced in \cite{JS85} and is called \emph{stability}. A countable ergodic probability measure preserving (pmp) equivalence relation $\cR$ is called \emph{stable} if $\cR \cong \cR \times \cR_0$, where $\cR_0$ is the unique hyperfinite ergodic pmp equivalence relation. In \cite[Theorem 3.4]{JS85}, stability is characterized in terms of central sequences. Clearly, if $\cR$ is stable, the II$_1$ factor $L(\cR)$ is McDuff.

Again, both the McDuff property and stability are closely related to inner amenability. If $G$ is a countable group and $G \actson (X,\mu)$ is a free ergodic pmp action $(X,\mu)$, we consider the group measure space II$_1$ factor $M = L^\infty(X) \rtimes G$ and the orbit equivalence relation $\cR = \cR(G \actson X)$. By \cite{JS85}, if $\cR$ is stable, the group $G$ is inner amenable. Actually, already if $M$ is McDuff, $G$ must be inner amenable (see Proposition \ref{prop.McDuff-inner} below). We say that a group $G$ is \emph{stable}, resp.\ \emph{McDuff}, if $G$ admits a free ergodic pmp action $G \actson (X,\mu)$ such that the orbit equivalence relation $\cR(G \actson X)$ is stable, resp.\ the crossed product $L^\infty(X) \rtimes G$ is McDuff. Here and in what follows all probability spaces are assumed to be standard.

We thus have the following subtly related notions for a countable group $G$~: inner amenability, property Gamma, the McDuff property and stability. Figure \ref{fig:summary} summarizes all implications and non implications between these properties. The main contribution of this article is to fill in all the arrows that were unknown so far. So, we prove that Kida's counterexamples of \cite{Ki12a} provide icc groups with property Gamma that are not McDuff. We also adapt the example of \cite{PV08} of a (necessarily non icc) property~(T), but yet McDuff group $G$ to obtain an icc McDuff group that is not stable.

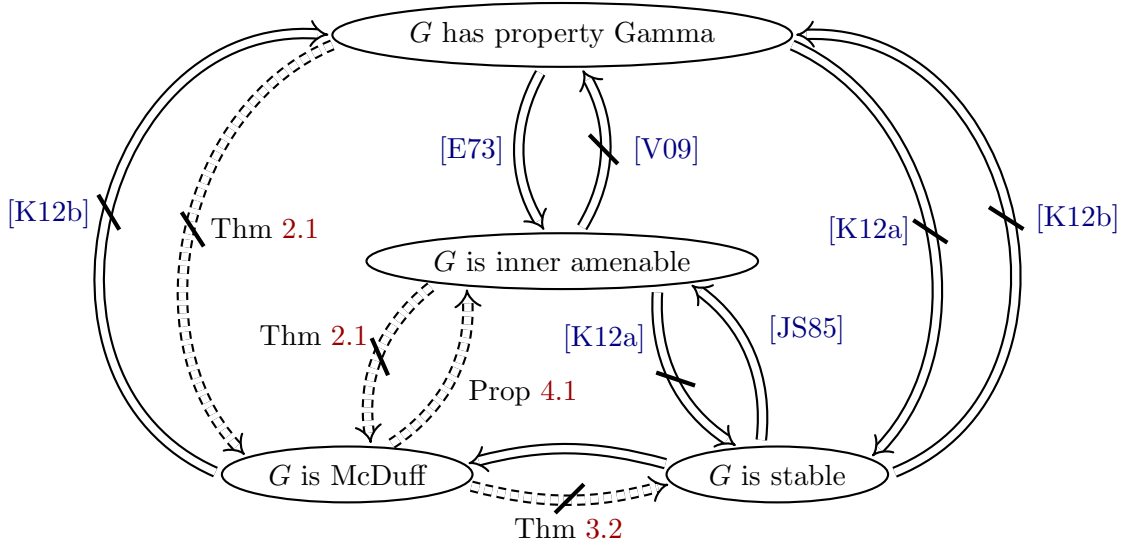
\begin{figure}[h]
        \centering
        \begin{tikzpicture}

        \draw (0,0) node[property] (inner) {$G$ is inner amenable};
        \draw (inner) ++(90:3) node[property] (Gamma) {$G$ has property Gamma};
        \draw (inner) ++(-45:4) node[property] (stable) {$G$ is stable};
        \draw (inner) ++(-135:4) node[property,align=center] (McDuff) {$G$ is McDuff};

        \draw[rightarrow] (Gamma) to[bend right] node[anchor=east,outer sep=0.2em]{\cite{Ef73}} (inner);
        \draw[rightarrow strike thru] (inner) to[bend right] node[anchor=west, outer sep=0.3em]{\cite{Va09}}  (Gamma);

        \draw[rightarrow,densely dashed] (McDuff.35) to[bend right] node[below right = -0.3em and 0.25em] {Prop \ref{prop.McDuff-inner}} (inner.-165);
        \draw[rightarrow strike thru,densely dashed] (inner.-170) to[bend right] node[anchor=south east] {Thm \ref{thm.Gamma-not-McDuff}} (McDuff.50);

        \draw[rightarrow] (stable.120) to[bend right] node[anchor=south west, outer sep=0em]{\cite{JS85}} (inner.-10);
        \draw[rightarrow strike thru] (inner.-15) to[bend right] node[above left=0.5em and 0.6em]{\cite{Ki12a}} (stable.145);

        \draw[rightarrow strike thru,densely dashed] (McDuff.-5) to[bend right=15] node[anchor=north, outer sep=0.3em] {Thm \ref{thm.McDuff-not-stable}} (stable.185);
        \draw[rightarrow] (stable.175) to[bend right=15] (McDuff.5);

        \draw[rightarrow strike thru,densely dashed] (Gamma.182) to[bend right=55] node[anchor=west,outer sep=0.3em] {Thm \ref{thm.Gamma-not-McDuff}} (McDuff.170);
        \draw[rightarrow strike thru] (Gamma.-2) to[bend left=55] node[anchor=east,outer sep=0.3em] {\cite{Ki12a}} (stable.10);

        \draw[rightarrow strike thru] (McDuff.180) to[bend left=80,looseness=1.3] node[anchor=east,outer sep=0.3em] {\cite{Ki12b}} (Gamma.180);
        \draw[rightarrow strike thru] (stable.0) to[bend right=80,looseness=1.3] node[anchor=west,outer sep=0.3em] {\cite{Ki12b}} (Gamma.0);
        \end{tikzpicture}
        \caption{Implications and non implications for a countable discrete icc group $G$. The dashed arrows are proven in this article.}
        \label{fig:summary}
    \end{figure}

For later use, recall from \cite[Theorem 1.2]{Jo99} that a subgroup $H$ of a countable group $G$ has \emph{relative property (T)} if and only if for every unitary representation $\pi : G \recht \cU(\cK)$ and every bounded sequence of vectors $\xi_n \in \cK$ satisfying $\lim_n \|\pi(g)\xi_n - \xi_n\| = 0$ for all $g \in G$, we have that $\lim_n \|P_H(\xi_n) - \xi_n\| = 0$, where $P_H : \cK \recht \cK^H$ denotes the orthogonal projection of $\cK$ onto the closed subspace of $\pi(H)$-invariant vectors.

In the context of tracial von Neumann algebras, this characterization of relative property~(T) gives the following well known lemma. For completeness, we include a proof.

\begin{lemma}\label{lem.char-prop-T}
Let $(M,\tau)$ be a tracial von Neumann algebra with von Neumann subalgebra $N \subset M$. Let $G$ be a countable group and $H < G$ a subgroup with the relative property~(T). Let $\pi : G \recht \cU(M)$ be a homomorphism satisfying $\pi(g) N \pi(g)^* = N$ for all $g \in G$. Define $B = \pi(H)' \cap N = \{ x \in N \mid \pi(h) x = x \pi(h) \; \forall h \in H\}$. Denote by $E_B : N \recht B$ the unique trace preserving conditional expectation.

If $(a_n)$ is a bounded sequence in $N$ satisfying $\lim_n \|\pi(g) a_n - a_n \pi(g)\|_2 = 0$ for all $g \in G$, then $\lim_n \|a_n - E_B(a_n)\|_2 = 0$.
\end{lemma}
\begin{proof}
It suffices to observe that $(\Ad \pi(g))_{g \in G}$ defines a unitary representation of $G$ on $L^2(N)$ and that $E_B$ is the orthogonal projection of $L^2(N)$ onto the subspace of $H$-invariant vectors.
\end{proof}

\section{An icc group with property Gamma that is not McDuff}

In \cite{Ki12a}, Kida constructed the first icc groups $G$ that have property Gamma (and in particular, are inner amenable) but that are not stable, meaning that they do not admit a stable free ergodic pmp action. We prove here that these groups are not even McDuff, meaning that they do not admit a free ergodic pmp action with the crossed product being McDuff.

Fix a property~(T) group $H$ that admits a central element $a \in \cZ(H)$ of infinite order. As in \cite{dC05}, one could take $n \geq 3$, define $H \subset \SL(n+2,\Z)$ given by
$$H = \left\{ \; \begin{pmatrix} 1 & b & c \\ 0 & A & d \\ 0 & 0 & 1 \end{pmatrix} \; \middle| \; \parbox[c]{4.5cm}{$\dis A \in \SL(n,\Z) \; , \; b \in \Z^{1 \times n} \; , \vspace{0.9ex}\\ c \in \Z \; , \; d \in \Z^{n \times 1}$} \right\} \quad\text{and take}\quad a = \begin{pmatrix} 1 & 0 & 1 \\ 0 & \bI_n & 0 \\ 0 & 0 & 1 \end{pmatrix} \; .$$

Given two non zero integers $p,q$ with $|p| \neq |q|$, we define $G$ as the HNN extension
\begin{equation}\label{eq.G-Gamma-not-McDuff}
G = \langle H, t \mid t^{-1} a^p t = a^q \rangle \; .
\end{equation}

\begin{theorem}\label{thm.Gamma-not-McDuff}
Kida's group $G$ defined in \eqref{eq.G-Gamma-not-McDuff} is icc, has property Gamma, but is not McDuff.
\end{theorem}
\begin{proof}
Analyzing reduced words in HNN extensions, it is easy to check that $G$ is icc and that $H < G$ is a relatively icc subgroup, meaning that $\{h g h^{-1} \mid h \in H\}$ is infinite for every $g \in G \setminus H$ (see e.g.\ \cite{St05}).

Ozawa already proved that $L(G)$ has property Gamma and his proof can be found in \cite[Theorem 1.4]{Ki12a}. For completeness, we include the following slightly easier proof. We claim that there exists a sequence of Borel functions $f_n : \R \recht \T$ satisfying $f_n(x+1) = f_n(x)$ for all $x \in \R$,
\begin{equation}\label{eq.claim-fn}
\lim_n \int_0^1 f_n(x) \; dx = 0 \quad\text{and}\quad  \lim_n \int_0^1 |f_n(px) - f_n(qx)|^2 \; dx = 0 \; .
\end{equation}
Denote by $(u_g)_{g \in G}$ the canonical unitary elements of $L(G)$. Given such a sequence of functions $f_n$ and denoting by $\widehat{f_n}(k) = \int_0^1 f_n(x) e^{-2\pi i k x} \; dx$ the Fourier coefficients of $f_n$, we can define the unitary elements $u_n \in L(H)$ given by
$$u_n = \sum_{k \in \Z} \widehat{f_n}(k) \, u_{a}^{pk} \; .$$
Then \eqref{eq.claim-fn} says that $\tau(u_n) \recht 0$ and $\|u_t^* u_n u_t - u_n \|_2 \recht 0$. Since moreover $u_n \in \cZ(L(H))$, it follows that $(u_n)$ is a non trivial central sequence in $L(G)$, so that $L(G)$ has property Gamma.

To prove the existence of $f_n$ satisfying \eqref{eq.claim-fn}, define $\Lambda = \Z[p^{-1},q^{-1}] \rtimes \Z$ as the semidirect product of the abelian group $\Z[p^{-1},q^{-1}]$ with $\Z$ acting by the automorphisms given by multiplication with a power of $p/q$. Define the action $\Lambda \actson^\al \R$ given by $(x,k) \cdot y = x + (p/q)^k y$ and equip $\R$ with the Lebesgue measure class. Since $\Lambda$ is an amenable group, it follows from \cite{CFW81} that the orbit equivalence relation $\cR(\Lambda \actson \R)$ is hyperfinite and thus, not strongly ergodic. This provides us with a sequence of unitary elements $f_n \in L^\infty(\R)$ tending to $0$ weakly and being approximately $\Lambda$-invariant in the sense that $\al_{(x,k)}(f_n) - f_n \recht 0$ strongly for all $(x,k) \in \Lambda$. Restricting $f_n$ to the interval $[0,1)$ and then extending $f_n$ periodically, we find a sequence satisfying \eqref{eq.claim-fn}.

Choose a free ergodic pmp action $G \actson (X,\mu)$. We prove that the II$_1$ factor $M = L^\infty(X) \rtimes G$ is not McDuff. Let $u_n \in \cU(M)$ be a central sequence. Put $B = L(H)' \cap M$. Since $H$ has property~(T), by Lemma \ref{lem.char-prop-T}, we have that $\|u_n - E_B(u_n)\|_2 \recht 0$, where $E_B : M \recht B$ is the unique trace preserving conditional expectation. Since $H < G$ is relatively icc, it follows that $B \subset B_1$, where $B_1 = L^\infty(X) \rtimes H$. We also define $B_2 = L^\infty(X) \rtimes t H t^{-1}$. Since $\|u_n - E_{B_1}(u_n)\|_2 \recht 0$ and $\|u_n - u_t^* u_n u_t\|_2 \recht 0$, we get that $\|u_n - E_{B_2}(u_n)\|_2 \recht 0$. Since $H \cap t H t^{-1} \subset \cZ(H)$, we conclude that $\|u_n - E_D(u_n)\|_2 \recht 0$, where $D = L^\infty(X) \rtimes \cZ(H)$.

Define $C = L^\infty(X)^H \rtimes \cZ(H)$, where $L^\infty(X)^H$ denotes the von Neumann algebra of $H$-invariant functions. We claim that the subalgebras $B, D \subset M$ form a commuting square with $B \cap D = C$. To prove this claim, it suffices to take $a \in B$ and prove that $E_D(a) \in C$. Let $a = \sum_{g \in G} a_g u_g$ with $a_g \in L^\infty(X)$ for all $g \in G$ be the Fourier decomposition of $a$. Since $a$ commutes with $L(H)$, we get in particular that $a_h \in L^\infty(X)^H$ for all $h \in \cZ(H)$. Since $E_D(a) = \sum_{h \in \cZ(H)} a_h u_h$, the claim follows. Since $\|u_n - E_B(u_n)\|_2 \recht 0$ and $\|u_n - E_D(u_n)\|_2 \recht 0$, we conclude that $\|u_n - E_C(u_n)\|_2 \recht 0$. Because $C$ is abelian, $M$ is not McDuff.
\end{proof}

\section{An icc McDuff group that is not stable}

There are two obvious obstructions for stability of a group $G$~: non inner amenability and property~(T). In \cite[Theorem 6.4.2]{PV08} using \cite{Er10}, an example of a McDuff property~(T) group $G$ is constructed. In particular, $G$ is an example of a McDuff group that is not stable. However by \cite[Theorem 6.4.1]{PV08}, a McDuff property~(T) group can never be icc.

In this section, we construct examples of icc McDuff groups that are not stable. We do this by giving the following obstruction for stability and by combining the methods of \cite{PV08} and \cite{Va09}.

\begin{proposition}\label{prop.rel-icc-T-not-McDuff}
If a countable group $G$ admits a subgroup $H < G$ with the relative property~(T) and the relative icc property, then $G$ is not stable.
\end{proposition}
\begin{proof}
Let $G \actson (X,\mu)$ be a free ergodic pmp action and denote by $\cR = \cR(G \actson X)$ the associated orbit equivalence relation. Write $A = L^\infty(X)$ and $M = A \rtimes G$. We denote by $\cN_M(A) = \{u \in \cU(M) \mid u A u^* = A\}$ the normalizer of $A$ inside $M$. To prove that $\cR$ is not stable, by \cite[Theorem 3.4]{JS85}, we have to show that all central sequences $(u_n)$ and $(p_n)$ in $M$ with $u_n \in \cN_M(A)$ and $p_n$ a projection in $A$ for all $n$ satisfy $\|u_n p_n - p_n u_n\|_2 \recht 0$.

Put $B = L(H)' \cap M$ and also define $D = A^H$ as the algebra of $H$-invariant functions in $A$. Denote by $E_B : M \recht B$ and $E_D : A \recht D$ the unique trace preserving conditional expectations. Since $H < G$ has the relative property~(T), by Lemma \ref{lem.char-prop-T}, we get that $\|u_n - E_B(u_n)\|_2 \recht 0$ and $\|p_n - E_D(p_n)\|_2 \recht 0$. Since $H < G$ is relatively icc, we have $B \subset A \rtimes H$. Therefore, $B$ and $D$ commute. It follows that $\|u_n p_n - p_n u_n\|_2 \recht 0$.
\end{proof}

For every countable group $H$, we denote by $H_f$ its FC-radical, defined as the normal subgroup of $H$ consisting of all elements that have a finite conjugacy class. Also recall that a group is said to be virtually abelian if it admits an abelian subgroup of finite index. We denote by $C_H(K)$ the centralizer in $H$ of a subset $K \subset H$.

To construct an icc McDuff group that is not stable, we start with the property~(T) group of \cite[Theorem 6.4]{PV08}~: a residually finite property~(T) group $H$ such that the FC-radical $H_f$ is not virtually abelian. In \cite{Er10}, it was proved that such a group indeed exists. From \cite[Theorem 6.4]{PV08}, we know that $H$ is McDuff. We then perform the same iterative amalgamated free product construction as in \cite{Va09} in order to embed $H$ as a relatively icc subgroup of a larger group $G$ in such a way that $G$ remains McDuff.

Fix a decreasing sequence of finite index normal subgroups $L_n \lhd H$ with $\bigcap_n L_n = \{e\}$. Enumerate $H_f = \{h_n \mid n \in \N\}$. Note that for every $n \in \N$, $C_H(h_1,\ldots,h_n)$ is a finite index subgroup of $H$. Then,
$$H_n := L_n \cap \left( \bigcap_{g \in H} g \, C_H(h_1,\ldots,h_n) \, g^{-1} \right)$$
is a decreasing sequence of finite index normal subgroups of $H$ with $\bigcap_n H_n = \{e\}$. Put $K_n = H_f \cap H_n$.

Define inductively $H = G_0 \subset G_1 \subset G_2 \subset \cdots$ as the amalgamated free product $G_{n+1} = G_n *_{K_n} (K_n \times \Z)$, where we view $K_n \subset H = G_0 \subset G_n$. Define $G$ as the direct limit of $G_0 \subset G_1 \subset \cdots$.

\begin{theorem}\label{thm.McDuff-not-stable}
The group $G$ constructed above is icc, has the McDuff property, but is not stable.
\end{theorem}
\begin{proof}
We first prove that $H < G$ is relatively icc. Take $g \in G \setminus H$. Take $n \geq 0$ such that $g \in G_{n+1} \setminus G_n$. We have $G_{n+1} = G_n *_{K_n} (K_n \times \Z)$ and $K_n < H < G_n$. Since $H_f$ is infinite amenable and $H$ has property~(T), $H_f < H$ has infinite index. So $K_n$ has infinite index in $H$ and it follows that $\{h g h^{-1} \mid h \in H\}$ is an infinite set. To prove that $G$ is icc, it suffices to prove that every $h \in H$ with $h \neq e$ has an infinite conjugacy class. Take $n \geq 0$ such that $h \in H \setminus K_n$. Let $t$ be the generator of $\Z$ in the description $G_{n+1} = G_n *_{K_n} (K_n \times \Z)$. Since $h \in G_n \setminus K_n$, the elements $t^k h t^{-k}$, $k \in \Z$, are all distinct.

Since $H$ has property~(T), it follows from Proposition \ref{prop.rel-icc-T-not-McDuff} that $G$ is not stable. It remains to construct a free ergodic pmp action $G \actson (X,\mu)$ such that $L^\infty(X) \rtimes G$ is McDuff.

For every $n \geq 0$, we have $G_{n+1} = G_n *_{K_n} (K_n \times \Z)$ and we define the homomorphism $\pi_{n,n+1} : G_{n+1} \recht G_n$ given by the identity map on $G_n$ and by mapping $\Z$ to $\{e\}$. Writing $\pi_{n,m} = \pi_{n,n+1} \circ \cdots \circ \pi_{m-1,m}$ for all $n \leq m$, we obtain a compatible system of homomorphisms $\pi_{n,m} : G_m \recht G_n$ that combine into a homomorphism $\pi_n : G \recht G_n$. In particular, $\pi_0 : G \recht H$. By induction on $n$, one checks that $K_m$ is a normal subgroup of $G_{n+1}$ for all $n \geq 0$ and all $m \geq n$, and that $g \pi_0(g)^{-1}$ commutes with $K_n$ for every $g \in G_{n+1}$.

Denote by $H \actson^{\be} (Y,\nu)$ the profinite action given as the inverse limit of $H \actson H/H_n$. Define $G \actson^{\al} Y$ by $\al(g) = \be(\pi_0(g))$. Note that $\al$ is an ergodic pmp action. For every $n \geq 1$, we consider the Bernoulli action $\beta_n$ of $G_{n}/K_{n-1}$ on $X_n = [0,1]^{G_{n}/K_{n-1}}$ equipped with the product measure $\mu_n = \lambda^{G_n/K_{n-1}}$. We define $G \actson^{\al_n} X_n$ by $\al_n(g) = \be_n(\pi_n(g))$. Note that $\al_n$ is a weakly mixing pmp action. We define $(X,\mu)$ as the product of $(Y,\nu)$ and all $(X_n,\mu_n)$, $n \geq 1$. We define $G \actson X$ as the diagonal action. Since each $\al_n$ is weakly mixing and $\al$ is ergodic, the diagonal action $G \actson (X,\mu)$ is ergodic. When $g \in G \setminus \{e\}$, take $n \geq 1$ such that $g \in G_n \setminus K_{n-1}$. Since the Bernoulli action $\beta_n$ is essentially free, the set $\{x \in X_n \mid \al_n(g)(x) = x \}$ has measure zero and thus also $\{x \in X \mid g\cdot x = x\}$ has measure zero. We conclude that $G \actson X$ is essentially free.

Denote $M = L^\infty(X) \rtimes G$ and let $\tau : M \recht \C$ be the tracial state on $M$. To prove that $M$ is McDuff, we use the method of \cite[Theorem 6.4]{PV08}. Denote by $\theta : L^\infty(Y) \recht L^\infty(X)$ and $\theta_n : L^\infty(X_n) \recht L^\infty(X)$ the natural $*$-homomorphisms induced by the coordinate maps $X \recht Y$ and $X \recht X_n$. The definition of $H_k$ implies that for every $h \in H_f$, we have $H_k \subset C_H(h)$ for all $k$ large enough. So for every $h \in H_f$, we can define the element
$$v_h = \sum_{s \in H/H_k} \theta(1_{sH_k}) \, u_{s h^{-1} s^{-1}}$$
whenever $k$ is such that $H_k \subset C_H(h)$. A direct computation shows that $\tau(v_h)=0$ for every $h \in H_f \setminus \{e\}$ and that $h \mapsto v_h$ is a homomorphism from $H_f$ to $\cU(L(H)' \cap M)$.

When $g \in G_{n+1}$, we have that $g \pi_0(g)^{-1}$ commutes with $K_n$. It follows that $v_h$ commutes with $L(G_{n+1})$ whenever $h \in K_n$. When $h \in K_n$, we also get that $v_h$ commutes with $\theta(1_{r H_m})$ for all $r \in H$ and all $m \leq n$ and that $v_h$ commutes with $\theta_m(L^\infty(X_m))$ for all $m \leq n+1$. So, choosing any sequence $h_n \in K_n \setminus \{e\}$, we get that $(v_{h_n})_n$ is a central sequence of unitaries in $M$ with $\tau(v_{h_n}) = 0$ for all $n$. Since $K_n$ is a finite index subgroup of $H_f$, we know that $K_n$ is non abelian. We can thus choose sequences $h_n, h'_n \in K_n$ such that $h_n h'_n \neq h'_n h_n$ for all $n$. The corresponding central sequences $(v_{h_n})$ and $(v_{h'_n})$ satisfy $\tau(v_{h_n} v_{h'_n} v_{h_n}^* v_{h'_n}^*) = 0$ for all $n \geq 0$. So, $M$ is a McDuff II$_1$ factor.
\end{proof}

\section{All McDuff groups are inner amenable}

For icc groups $G$, the following result is a consequence of the main theorem in \cite{Ch81}. In the non icc case, we need a small extra argument, because we follow the convention that a group $G$ is inner amenable if there exists a conjugation invariant mean on $G$ that gives weight zero to all finite subsets of $G$ and not only to the subset $\{e\}$.

\begin{proposition}\label{prop.McDuff-inner}
Every McDuff group $G$ is inner amenable.
\end{proposition}
\begin{proof}
Let $G$ be a group that is not inner amenable. Let $G \actson^\al (X,\mu)$ be any free ergodic pmp action. Denote $A = L^\infty(X)$ and $M = A \rtimes G$. Let $x_n \in \cU(M)$ be an arbitrary central sequence. We prove that $\lim_n \|x_n - E_A(x_n)\|_2 = 0$.

Write $x_n = \sum_{g \in G} a_n(g) \, u_g$ with $a_n(g) \in A$ for every $g \in G$. Fix $g \in G \setminus \{e\}$. Since $\lim_n \|p \, x_n - x_n \, p\|_2 = 0$ for all $p \in A$, we have $\lim_n \|p \, a_n(g) - a_n(g) \, \al_g(p) \|_2 = 0$ for all $p \in A$. Let $p_1 \in A$ be a maximal projection satisfying $p_1 \al_g(p_1) = 0$. Put $p_2 = \al_g(p_1)$ and $p_3 = 1-(p_1+p_2)$. Since the action $\al$ is essentially free, we get that $\sum_i p_i = 1$ and $p_i \, \al_g(p_i) = 0$ for all $i$. It follows that $\lim_n \|p_i \, a_n(g)\|_2 = 0$ for every $i$, so that $\lim_n \|a_n(g)\|_2 = 0$ for all $g \in G \setminus \{e\}$.

Define the unit vectors $\xi_n \in \ell^2(G)$ given by $\xi_n(g) = \|a_n(g)\|_2$. In the previous paragraph, we proved that $\lim_n \xi_n(g) = 0$ for every $g \in G \setminus \{e\}$. As in \cite{Ch81}, the vectors $\xi_n$ are approximately conjugation invariant. Since $G$ is not inner amenable, the functions $\xi_n$ must satisfy $\lim_n \xi_n(e) = 1$, meaning that $\lim_n \|x_n - E_A(x_n)\|_2 = 0$.

So, we have proved that every central sequence $x_n \in \cU(M)$ satisfies $\lim_n \|x_n - E_A(x_n)\|_2 = 0$. It follows that $M' \cap M^\omega$ is abelian, so that $M$ is not McDuff.

\end{proof}

\end{document}